\documentclass[12pt]{extarticle}
\usepackage{amsmath, amsthm, amssymb, hyperref, color}
\usepackage[capitalise]{cleveref}
\usepackage[shortlabels]{enumitem}
\usepackage{graphicx}
\usepackage{indentfirst}
\usepackage{makecell}
\usepackage[final]{pdfpages}
\setboolean{@twoside}{false}
\usepackage{pdfpages}
\usepackage{tabularx}
\usepackage{caption}
\usepackage{subcaption}
\usepackage{chemfig, chemnum}
\usepackage{scalefnt}
\usepackage{verbatim}
\usepackage{footmisc}
\oddsidemargin \evensidemargin
\newtheorem{theorem}{Theorem}
\numberwithin{theorem}{section}
\newtheorem{proposition}[theorem]{Proposition}
\newtheorem{lemma}[theorem]{Lemma}

\newtheorem{definition}[theorem]{Definition}
\newtheorem{problem}[theorem]{Problem}
\newtheorem{remark}[theorem]{Remark}
\newtheorem{example}[theorem]{Example}
\newtheorem{conjecture}[theorem]{Conjecture}

\newcommand{\RR}{\mathbb{R}}

\theoremstyle{example}

\newtheoremstyle{example_contd}
{\topsep} {\topsep}%
{\upshape}
{}
{\bfseries\scshape}
{.}
{1em}
{\thmname{#1} \thmnumber{ #2}\thmnote{#3} (continued)}

\theoremstyle{example_contd}
\newtheorem*{example_contd}{Example}

\DeclareMathOperator{\intr}{int}
\DeclareMathOperator{\conv}{conv}

  \date{}

\title{\textbf{Attainable Regions of Dynamical Systems}}

\author{Nidhi Kaihnsa}

\begin{document}
\maketitle

\begin{abstract} \noindent We present a mathematical definition for the attainable region of a
dynamical system, with primary focus on mass action kinetics for chemical
reactions. We characterise this region for linear dynamical systems, and
we report on experiments and conjectures for weakly reversible systems
with linkage class one. A construction due to Vinzant is adapted to give
a representation of faces in the convex hull of trajectories.

\end{abstract}
\let\thefootnote\relax\footnotetext{\noindent \footnotesize {\bf Author's address:} \smallskip \\
\noindent
Max Planck Institute for Mathematics in the Sciences,
Inselstra\ss e 22, 04103 Leipzig, Germany. \ \ \

\noindent Email address: \texttt{Nidhi.Kaihnsa@mis.mpg.de} }
\section{Introduction}
A chemical reactor is a system where a set of reactions and their 
mixing takes place. Chemical engineers are interested in finding the most cost-efficient reactor for a given chemical reaction. Fritz Horn in 1964 was the first to reduce this optimisation problem to that of 
finding the feasible set of the optimisation problem~\cite{Horn}. He called this feasible set the $\textit{attainable region}$ of a system. By definition, this region is the set of all realisable states of the reaction network in question with a certain starting point. Over the past half a 
century, the optimisation problem in the field of chemical reaction networks has 
been foremostly developed by Martin Feinberg \cite{Fei, Fei02}, Roy 
Jackson \cite{HJ}, David Glasser, and Diane Hildebrandt \cite{MGHGM}. More recently, the field of 
chemical reaction networks has gathered a lot of attention in the mathematics 
community \cite{Bor, CDSS, DJFN, JS}, and is a fast growing field. 
\medskip 

In this article, we formalise the definition of attainable regions and characterise them for some special systems using well known notions from algebraic geometry. To the best of the author's knowledge, this is first rigorous mathematical treatment of attainable regions. We aim for our contribution to help towards better understanding of the convex hulls of trajectories of dynamical systems.
\medskip

In the subsequent Section~2 we first set up the basic notation and define the attainable region for a general chemical reaction network. In Section~3 we then characterise the attainable region for linear systems. We show that for linear systems the convex hull of the trajectories are the attainable regions. In particular, we show that the feasible set of the reactor-optimisation problem for a class of linear systems can be expressed as the feasible set of a semidefinite program---using the language of algebraic geometry, the attainable region is a 
{\em spectrahedral shadow}. We then move on to Section~4 where we discuss a number of computational experiments on weakly reversible systems with a single linkage class, so chemical reaction 
networks given by a strongly connected digraph. These experiments enable us to formulate a new conjecture about attainable regions in the non-linear case.
The attainable region is a convex object and to understand this convex object we 
would like to understand its faces.
In Section~5, we used one such approach to 
understand the faces of the convex hull of the trajectories of 
weakly reversible systems with their convex body each in dimension 3, 4 and~5. The article ends with a discussion of our results and an outlook into future applications of this new rigorous treatment of chemical reaction networks. 

\section{Notation}

Throughout this article we follow the standard notation from chemistry and
denote \textit{chemical species} by $X_1, X_2, \dots, X_s$ for some $s\in \mathbb{N}$. Each of these 
species has a concentration $x_1(t), 
x_2(t), \dots, x_s(t) \in \RR^{\geq 0}$, respectively, at any time $t$ for $t\in [0,\infty)$. A {\em chemical complex} is a linear combination with non-negative integer coefficients of chemical species.
As defined in \cite{CDSS}, a \textit{chemical reaction network} (CRN) is a directed graph $G$ with vertex set 
$V=\{1,2,\dots, n \}$ and edge set $E \subseteq \{(i,j) \in V \times V \ : i\neq j
\}$. In such a graph the vertex 
$i\in V$ represents a chemical complex, and the edges indicate that a reaction takes place from one 
complex to the other. In addition, the edges are weighted by their reaction rates.

\begin{example}\label{eg} \rm
The following figure shows a network of chemical reactions.

\begin{center}

\schemestart 
    \subscheme{$X_2$ + 2$X_5$}
    \arrow(be--ac){<=>[$\kappa_5$][$\kappa_6$]}[135]
    \subscheme{$X_1$ + $X_3$}
      \arrow(@be--d){<=>[$\kappa_3$][$\kappa_4$]}[45]
    \subscheme{$X_4$}
    \arrow(@ac--@d){<=>[$\kappa_1$][$\kappa_2$]}
\schemestop
\end{center}
\bigskip
Here, $X_i$ are the species for $i\in\{1,2,\dots, 5\}$. The chemical complexes are
\{$X_2+2X_5$\}, \{$X_1 +X_3$\}, and \{$X_4$\}. The labels 
$\kappa_i$ for $i \in\{1,2,\dots, 6\}$ are the corresponding rates of reactions.
$\hfill \square$
\smallskip
\end{example}
Given such a reaction network, we are interested in the evolution of the 
concentrations of the species over time, dictated by the mass-action kinetics. 
 For $s$ species and $n$ complexes in a network, let henceforth $Y=(y_{ij})$ denote the  $n\times s$ 
matrix with the entry $y_{ij}$ being the coefficient of the $j$-th species in the $i$-th complex. We associate with the vertex $i$ of a CRN the monomial
$x^{y_i}=x_1^{y_{i1}}x_2^{y_{i2}}\cdots x_s^{y_{is}}$. This is a simple transformation of the linear combination defining complexes in the CRN which enables us to write
the dynamics for the mass-action kinetics as
\begin{equation}\label{eqn:ds}
\dot{x}=\frac{dx}{dt}=\Psi(x)\cdot A_{\kappa}\cdot Y
\end{equation}
where $\Psi(x)=\begin{bmatrix}x^{y_1}& x^{y_2}&\cdots &x^{y_n}
\end{bmatrix}$ and $A_{\kappa}
=(\kappa_{ij})$ is a matrix with $ij$-th 
entry given by the rates of reactions from the $i$-th complex to the 
the $j$-th complex for $i\neq j$ and $\sum_j \kappa_{ij}=0$ for all $i$. This matrix is 
the negative of Laplacian of the weighted digraph $G$. \\
\begin{example_contd}[\ref{eg}]
In the network illustrated in \cref{eg}, the monomials corresponding to the complexes are $x_1x_3\text{, }x_4\text{ and }x_2x_5^2$ and hence,

\begin{gather*}
\Psi(x)=\begin{bmatrix}x_1x_3&x_4&x_2x_5^2\end{bmatrix}\text{, }
A_\kappa=\begin{bmatrix}
-\kappa_1-\kappa_5	 & \kappa_1 					& \kappa_5 \\
 \kappa_2 					& -\kappa_2-\kappa_4	 & \kappa_4\\
 \kappa_6					 & \kappa_3 					& -\kappa_6-\kappa_3
 \end{bmatrix}\text{, }
 Y=\begin{bmatrix}
1&0&1&0&0 \\
0&0&0&1&0\\
 0&1&0&0&2
 \end{bmatrix}.
 \end{gather*}
Using the notation established, the dynamics of the above network is given by the system of ODEs below
\begin{equation}
\begin{split}
\dot{x_1}&=\frac{dx_1}{dt}=(-\kappa_1 - \kappa_5) x_1 x_3 + \kappa_2 x_4 + \kappa_6 x_2 x_5^2 \\
\dot{x_2}&=\frac{dx_2}{dt}=\kappa_5 x_1 x_3 + \kappa_4 x_4 + (-\kappa_3 - \kappa_6) x_2 x_5^2\\
\dot{x_3}&=\frac{dx_3}{dt}=(-\kappa_1 - \kappa_5) x_1 x_3 + \kappa_2 x_4 + \kappa_6 x_2 x_5^2 \\
\dot{x_4}&=\frac{dx_4}{dt}=\kappa_1 x_1 x_3 + (-\kappa_2 - \kappa_4) x_4 + \kappa_3 x_2 x_5^2 \\
\dot{x_5}&=\frac{dx_5}{dt}=2 (\kappa_5 x_1 x_3 + \kappa_4 x_4 + (-\kappa_3 - \kappa_6) x_2 x_5^2).
\end{split}
\end{equation}
\end{example_contd}

Let $y_j$ be the vector given by the $j$-th row of the matrix $Y$. Consider the linear subspace in $\RR^s$ spanned by $y_j-y_i$ whenever $(i,j)\in E.$ This space is called the \emph{stoichiometry subspace} and we will henceforth denote it by $P.$
 For a given dynamical system we always denote the initial value of the system at time $t=0$ by $x_0=x(0)\in \RR^s_{> 0}$. The trajectory that starts at $x_0$ stays in the affine subspace $(x_0 + P)\cap \RR^s_{\geq 0}$.
 We call a subset 
 $S\subset \RR^s$ 
 \textit{forward closed} subset if the initial condition $x_0\in S$ holds for 
 the dynamical system then all future values are contained in the subset $S$. In formulae, we thus have that $x_0 \in S$ implies $x(t)\in S$ for all $t\geq 0$. In particular, the non-negative orthant of $\RR^s$ is forward closed.\\
  
  In this 
 work, we aim to characterise all the 
possible sets of the species concentration attainable from the continuous reaction, 
according to the dynamics, and mixing of the concentrations of the species at all 
times. This approach to the reactor optimisation problem has been explored and discussed in \cite{MGHGM}. We approach this problem by building on a new mathematical definition of this attainable region, and we study these regions for various kinds of dynamical systems. 
\begin{definition}\rm
 The {\em attainable region}, $\mathcal{A}(x_0)$ is the smallest convex forward 
closed
subset of $\RR^s$ that contains the point $x_0$. 
\end{definition}
By construction, the attainable 
region is a convex subset in the closed positive 
orthant of real space 
$\RR^s$ of the chemical species. In the section that follows we first discuss the attainable regions of linear dynamical systems.

\section{Linear systems}

A dynamical system as in (\ref{eqn:ds}) is called $\textit{linear}$ when $n=s$ and $Y$ is the identity 
 matrix. In this case each of the complexes is a different single-unit species.

\smallskip
\begin{example}\label{eglin} \rm
The following graph illustrates the linear system of three species.

\begin{center}
\schemestart
    \subscheme{$X_3$}
    \arrow(be--ac){<=>[$\kappa_{13}$][$\kappa_{31}$]}[135]
    \subscheme{$X_1$ }
      \arrow(@be--d){<=>[$\kappa_{32}$][$\kappa_{23}$]}[45]
    \subscheme{$X_2$}
    \arrow(@ac--@d){<=>[$\kappa_{12}$][$\kappa_{21}$]}
\schemestop

\end{center}
 \bigskip
 
For the purpose of illustration, let now $\kappa_{12}=6,\kappa_{21}=1,\kappa_{32}=6,\kappa_{23}=1,
\kappa_{13}=3,\kappa_{31}=3$. From (\ref{eqn:ds}), we can express the 
dynamics of this system as
$$
\begin{bmatrix}
\dot{x}_1 &
\dot{x}_2 &
\dot{x}_3
\end{bmatrix}= 
 \begin{bmatrix}
x_1&
x_2&
x_3
\end{bmatrix}\cdot \begin{bmatrix}
 -9 & 6 & 3\\
 1 &-2 & 1\\
 3 & 6 & -9
 \end{bmatrix} 
$$ 

If $A_{\kappa}$ is diagonalisable, the solution to such a system is given by 
 \begin{equation}
 \label{eq:linsol}
{x(t)}=\sum_{k=1}^n ({x}_0 \cdot {r}_k){l}_k \exp(\lambda_k t)
 \end{equation} 
   where ${l}_k$ and ${r}_k$ are the left and right eigenvectors of $A_{\kappa}$ 
   corresponding 
   to eigenvalues $\lambda_k$, respectively,
 and 
 ${x}_0$ is the intial vector: for details see page 11 of \cite{CK}.\\
 This gives 
 $$
 \begin{bmatrix}
 x_1&
 x_2&
 x_3
 \end{bmatrix} =
 \begin{bmatrix}
 9/4e^{-8t}-3/2e^{-12t}+5/4\\ -9/2e^{-8t}+15/2\\ 9/4e^{-8t}+3/2e^{-12t}+5/4\\
 \end{bmatrix}^\top$$
with ${x}_0=\begin{bmatrix} 2 & 3 & 5 \end{bmatrix}$ as the starting 
 vector. 
For $t=0$,  we see that $\begin{bmatrix}
 x_1&
 x_2&
 x_3
 \end{bmatrix} = \begin{bmatrix} 2 & 3 & 5 \end{bmatrix}$ and as $t\rightarrow \infty$, this system continuously travels to 
 the stable point $\begin{bmatrix} 5/4 & 15/2 & 5/4 \end{bmatrix}$. On implicitizing the parametric equation in $t$, we obtain
\begin{equation}\label{eq:lintraj}
x_1+x_2+x_3-10 =0 \text{ and } 8x_2^3-99x_2^2+324x_2x_3+324x_3^2-270x_2-3240x_3+4725 = 0.
\end{equation}
 
These two equations fully describe the trajectory of the linear system from $x_0$ to the 
stable point on the plane cut out by $x_1+x_2+x_3-10 =0.$ A similar observation can be made for the system with a different starting point. We will prove later 
that the convex hull of this curve is the attainable region and this region can also be expressed as a so-called
spectrahedral shadow.
$\hfill \square$
\end{example}
 \smallskip
We henceforth denote by $C$ the solution of a dynamical system. This is the trajectory of the dynamics. In \cref{eglin} the trajectory is given by \cref{eq:lintraj} restricted from $x_0$ to the stable point.
The convex hull, $S=\conv (C)$, of $C$ is the smallest convex set in the concentration 
space $\RR^s$ containing the solution $C$. In the lemma below we can now show that for 
linear chemical reactions, the convex hull of the solution of the dynamics is 
forward closed. In words, for linear systems every point on any trajectory that starts with some point in the convex hull $S$ and follows the dynamics of the system is contained in this convex hull. 

\begin{lemma}\label{fc}
The convex hull of the trajectory of a linear dynamical system is forward closed.
\end{lemma}
\begin{proof}
Any point $c$ in the convex hull, $\conv (C)=S \subset \RR^s$, of the trajectory $C$ can be 
expressed as $c=\sum_i\mu_i c_i$, where $c_i$ are points on the trajectory, $
\mu_i\geq 0$, and $\sum_{i=1}^{s+1}\mu_i =1$ for 
$i \in  \{1,2,\cdots, s+1\}.$
First  let us consider the case where the Laplacian is diagonalisable as in \cref{eglin}. With $c$ as starting point, the new trajectory, as in (\ref{eq:linsol}), is given by 
\begin{equation}
 \label{eq:ls}
 {x(t)}=\sum_k \left(( \sum_i\mu_i c_i)\cdot {r}_k \right){l}_k 
 \exp(\lambda_k t)= \sum_i \mu_i \left(\sum_k ( c_i\cdot {r}_k ){l}_k 
 \exp(\lambda_k t)\right)
 \end{equation} 
is the convex sum of trajectories in $S$. Thus, $S$ is forward closed.

For the dynamical system $\dot{x}=x\cdot A_{\kappa}$ where $A_{\kappa}$ is not diagonalisable we perform a coordinate change by the matrix $U$ such that the matrix $UAU^{-1}$ is in its Jordan canonical form : see section 1.3 of \cite{CK}. 

It is enough to consider single Jordan block $J$. The solution of a single Jordan block form is given by $x(t)=x\text{ }U^{-1}\exp(tJ)\text{ }U$. Proceeding same as above with $c$ as the starting point 
\begin{equation}
\begin{split}
x(t)&=(\sum_i \mu_{i}c_i) \text{ }U^{-1}\exp(tJ)\text{ }U\\
&=(\sum_i \mu_{i}(x\text{ } U^{-1}\exp(t_iJ)\text{ }U))\text{ }U^{-1}\exp(tJ)\text{ }U\\
&=\sum_i \mu_{i}(x\text{ } U^{-1}\exp((t_i+t)J)\text{ }U)
\end{split}
 \end{equation} 
 This gives us that the convex hull of the trajectory of a linear dynamical system is forward closed.
\end{proof}

For the linear system with $x_0$ as the initial point, by \cref{fc} the attainable region $\mathcal{A}(x_0)$ is the convex hull of the trajectory. 

Next, we give a condition on the Laplacian of a linear reaction network for which the convex hull of the trajectory is a semi-algebraic set. A semi-algebraic set in $\RR^s$ is the solution set of finitely many polynomial inequalities as:
$\mathcal{S}=\{x\in \RR^s |\text{ } f_1(x)\geq 0, \ldots,f_n(x)\geq 0 \}$ where $f_i\in \RR[x_1,\ldots, x_s]$ for all $i\in {1,\ldots, n}$. These sets are very well understood objects in algebraic geometry and can sometimes be represented as a spectrahedral shadow \cite{Sch1}. 
A \textit{spectrahedral shadow} is a convex set
$S\subset \RR^m$ that can be expressed by a linear matrix inequality: 
$$
S=\{(x_1,x_2,\ldots,x_m)\in \RR^m| \text{ }\exists \text{ } (y_1,y_2,\ldots,y_p)\in \RR^p :A_0+\sum_ix_iA_i+
\sum_jy_jB_j \succcurlyeq 0\}
$$
where $A_0$, $A_i$ and $B_j$ are real symmetric $n\times n$ matrices for $i\in\{1,2,\ldots, m\}$, and $j\in\{1,2,\ldots, p \}$. We use the symbol $A\succcurlyeq 0$ to denote that the matrix $A$ is positive 
semidefinite. This is equivalent to $A$ having non-negative eigenvalues. In order to prove \cref{ss} we need the following useful fact on these semi-algebraic sets.

\begin{remark}\label{rem} \rm
Let $\phi:\RR^m\rightarrow\RR^n$ be an affine-linear map and $S\subset \RR^m$ be
a spectrahedral shadow. The linear image $\phi(S)\subset \RR^n$ is a spectrahedral shadow.
\end{remark} 
A spectrahedral shadow is a linear projection of the feasible set of a semidefinite 
program which is also called a $\textit{spectrahedron}$. Expressing the attainable region as a spectrahedral shadow has an 
advantage of getting good bounds for the 
optimisation of a linear objective function. 
 
 \begin{proposition}\label{ss}
 The convex hull of the trajectory of a linear chemical reaction network whose Laplacian has rational eigenvalues is a spectrahedral shadow.
 \end{proposition}
 \begin{proof}

 Consider a rational curve $C : I \longrightarrow \RR^n$ given by $t \mapsto 
 (t^{a_1},t^{a_2},\ldots,t^{a_n})$ in $\mathbb{R}^n$ over an interval $I \subset \RR
 $ where $a_i$ are positive rational numbers for $i\in \{1,2,\dots, n \}$.
 For $0\leq t\leq 1$ this is a semialgebraic set $S$ of dimension 1.
 By Theorem 6.1 in Claus Scheiderer's paper \cite{Sch}, the closure of convex hull of $S$ is a spectrahedral 
 shadow.
 
 If $a_i$'s are the rational eigenvalues of the Laplacian of a linear chemical reaction network then the trajectory of the dynamical system is the image of $S$ under the map
 $\phi : S\longrightarrow \mathbb{R}^s$ for $0\leq t\leq 1$, given by the matrix 
whose $i$-th column vector is given by the transpose of the row vector $(({x}_0 \cdot {r}_i){l}_i).$

 The convex hull of the trajectory of a linear chemical reaction network is the linear 
 image of convex hull of $S$ and 
 therefore, by \cref{rem} is a spectrahedral shadow.
  \end{proof}

 Using \cref{fc} and \cref{ss}, in the theorem below, we can characterise the
 class of linear system for which the attainable region is a spectrahedral shadow.

\begin{theorem}
The attainable region of linear chemical reaction networks whose Laplacian has rational eigenvalues is spectrahedral shadow.
\end{theorem}
\begin{proof}
From \cref{ss}, we know that the convex hull of the trajectory of a linear chemical 
reaction network whose Laplacian has rational eigenvalues is a spectrahedral 
shadow. Also, for linear dynamical systems the convex hull is forward closed by 
\cref{fc}. Therefore, the attainable region is the convex hull of the trajectory and 
is a spectrahedral shadow.
\end{proof}

For a linear system whose Laplacian has rational eigenvalues, we can hence obtain an exact expression of its attainable region as a spectrahedral shadow. This enables us to use powerful methods of real algebraic geometry to study the properties of these sets.

One future stream of research, which we will not pursue in this text, is an extension of the above result to linear systems whose Laplacian has real, rather than rational, eigenvalues. To the best of author's knowledge this is not yet known. A property of this type would be an important step towards understanding the attainable region of a general dynamical system.

\section{Weakly Reversible Chemical Reaction Networks}

Following \cite{CDSS}, a chemical reaction network is called {\em weakly reversible} if each connected component of the underlying connected graph is strongly connected. Following the usual terminology from graph theory, a directed graph is strongly connected if there is a directed path between any two of its vertices.
In this article we will restrict ourselves to weakly reversible systems whose underlying graph has only one strongly connected component. These are called {\em linkage class one} systems. For these systems we conjecture the following: 
\begin{conjecture}
For weakly reversible systems with linkage class one the convex hull of the trajectory reaching a positive stable point is forward closed.
\end{conjecture}
In order to provide the computational evidence for this conjecture we followed a two-step procedure, outlined below. All computations were performed using the freely available software $\mathtt{SAGE}$ \cite{Sg}.
\medskip
\paragraph{Step one} Given $n$ vertices, we generate a random 
digraph. This graph is usually not strongly connected. We then add edges randomly 
between the strongly connected components of the generated graph to make it 
strongly connected. To each 
vertex of the graph we associate a monomial in $s$ indeterminates upto 
a degree $d$. This represents the chemical complex at that vertex as introduced in Section~2. 
These monomials are the entries of a matrix $\Psi (x)$ and the powers in the monomials give the matrix $Y$ in (\ref{eqn:ds}).
We obtain the matrix $A_{\kappa}$ by assigning random positive edge 
weights. These three matrices now fully specify a random dynamical system for a weakly reversible CRN. 

We numerically integrate the obtained dynamical system in $\mathtt{SAGE}$ 
using the Runge-Kutta 4 method. In order for it to effectively integrate we keep the 
degree of monomials below~5. For higher 
values of $d$, one may use a higher order Runge-Kutta method for integration.
This integration gives us points that lie on the solution $C$ of the system. Because we want to make a statement about the convex hull of the trajectory, we 
now construct a polytope in dimension $s$ which is the convex hull of the points obtained. 
$\mathtt{SAGE}$ uses the cdd library for 
this.

In our computations, we computed $10,000$ points per trajectory. The tailing points 
are closer to each other than the initial points, so we tailored the set of points for 
which we compute the convex hull $S$. Using a random point $c$ in $S$ as the initial point, for the same system we integrate again to get a new set of points on the new trajectory $C^{\prime}$  and ask if  
$S$ contains 
the points on $C^{\prime}$. 
This was done for various trajectories in $\mathbb{R}^s$ for $s=2,3,4,5,6.$

During these computations we faced various challenges. Most of these pertained to 
the fact that the computations were numerical, and also, to the large number of 
points. In particular, the computations were not always feasible in dimensions higher than $s=6$. Computing the polytope becomes harder for a large set of points and this required us to tailor the set of points accordingly.

\paragraph{Step two} It was proven in \cite{DJFN} and elaborated upon again in \cite{Bor} that every weakly 
reversible chemical reaction network has at least one positive steady state. During our computations in step one, 
we observed all the systems to be converging to a steady state. Moreover, the trajectories 
starting from any interior point also converged to the $\textit{same}$ 
point. This may, however, be due to the fact that the random graph we generated almost always had single stationary point. This leads us to \cref{prob}.

 Since the computations were numerical, as the 
dimensions got higher it became difficult to compute the polytope for more than 
first 100 points. Therefore, in the 
second part of 
the computations, we attempted to double check
the points which in  step one of the computations were found to not be in the convex hull possibly due to error while integration or computing the convex hull of floating point numbers. 
The tailing points on $C^{\prime}$, although reported as not contained in $S$ for many instances, were found to be in the close range of some point 
on the starting trajectory. 
Secondly, since we had tailored our set of points we checked by changing the subset of points on $C$
for which we computed the convex hull. This new polytope reported in some instances to contain the points that 
were not contained in the first polytope.

From the various computations we have computational evidence, in at least lower dimensions, that for 
strongly connected graphs the convex hull of the trajectory is forward closed. These computations also compels us to ask the following question:

\begin{problem}\label{prob}
In a weakly reversible system with random edge weights and a random starting point, how likely is that the stoichiometry space has multistationary points?
\end{problem}
Or put differently, in the space of weakly reversible systems in given dimension $d$, how big is the space of systems that have multiple stationary points?
This seems a fairly hard question for a general weakly reversible systems and to 
date not much is known about this problem. Similar questions have been asked for 
one-dimensional stoichiometry space in \cite{JS}. In general, it would be useful to 
be able to characterise the systems that have multiple stationary points. Such a characterisation may give us 
insight into the systems where the convex hull of trajectories is not forward closed 
and the attainable region is greater than the convex hull.

\section{Facial Structure}
In the previous section, we conjectured that for chemical reaction networks given by strongly connected graphs, the attainable region is the convex hull of the trajectory. To understand this object using convex algebraic geometry it is imperative to study its {\em faces}. For parametrized curves, one such approach was suggested by Cynthia Vinzant in Section 5.2 of her PhD dissertation \cite{Vin}.  We give the details of this below.

Let $C$ be a parametrized curve given by $
\textbf{g}=(g_1(t), \dots , g_m(t))$ for $t\in \mathcal{D}$. Here, $\mathcal{D}\subseteq \mathbb{R}$ is a closed interval and $g_i(t)$ are univariate polynomials in $t$ for $i\in \{1,2,\ldots ,m\}$.
The $r$-th {\em face-vertex set} Face($r$) of the curve $C$ is defined to be 
$$\{(d_1,\dots, d_r)\in \mathcal{D}^r | \text{ } \textbf{g}(d_1),\dots , \textbf{g}(d_r)  
\text{ are the vertices of a face of the convex hull of $C$} \}.$$

For $p\leq r$, let $\{d_1,\ldots ,d_p\}\in \intr (\mathcal{D})$ be interior and 
$\{d_{p+1},\ldots ,d_r\}\in \partial\mathcal{D}$ be the boundary points. As $d_i$ varies in $\mathcal{D}$, the face-vertex set Face($r$) 
is always contained in the variety cut out by 
\begin{equation}\label{eq:facevset3}
\text{minors}\left( n+1,\begin{pmatrix}
               \,  1  & \ldots &    1  &    0  &  \ldots  & 0   \\
                \, \textbf{g}(d_1)  & \ldots &  \textbf{g}(d_r)  &   \textbf{g}^{\prime}(d_1)  &  \ldots  &   \textbf{g}^{\prime}(d_p)  
               
\end{pmatrix}
\right).
\end{equation}
This describes a variety in $\mathcal{D}^r$ that contains the set Face($r$) for the convex hull of $C$. 

In this section, we apply this approach to the dynamical systems and illustrate them in the examples below. This method has not been previously used to understand the convex hulls. Note that for any curve $C$, it is true that if $c_1,\ldots , c_r$ are points on the curve such that they define vertex set of some face of the convex hull of $C$ then
\begin{equation}\label{eq:facevset2}
\text{minors}\left( n+1,\begin{pmatrix}
               \,  1  & \ldots &    1  &    0  &  \ldots  & 0   \\
             \, c_1  & \ldots &  c_r  &  c^{\prime}_{1}   &  \ldots  &   c^{\prime}_{p}  
               
\end{pmatrix}
\right)
\end{equation}
 vanish where $c^{\prime}_{i}$ denote the tangent vector at that point. We will exploit this fact and give representation of the faces.

\bigskip
In our case, we only had points on the curve and this makes it difficult to express faces as a variety. For a curve in $s$ dimension we were able to look at the following cases:
\begin{itemize}
\item Face($\frac{s+1}{2}$) if $s$ is odd.
\item Face($\frac{s}{2}+1$) if $s$ is even.
\end{itemize}

The above two conditions make the matrix in (\ref{eq:facevset2}) a square matrix and the corresponding faces are then given by the vanishing of the determinant.
We used the software $\mathtt{Mathematica}$~\cite{Mat} to plot the sign of the determinant for all combinations of points on the curve. We illustrate this for curves in dimensions 3, 4 and 5 below. These curves are given by the ODE's which satisfy the condition in the following lemma due to \cite{HT}.
\medskip
\begin{figure}[t]
\centering
\begin{minipage}{.5\textwidth}
  \centering
 \includegraphics[width=.4\linewidth]{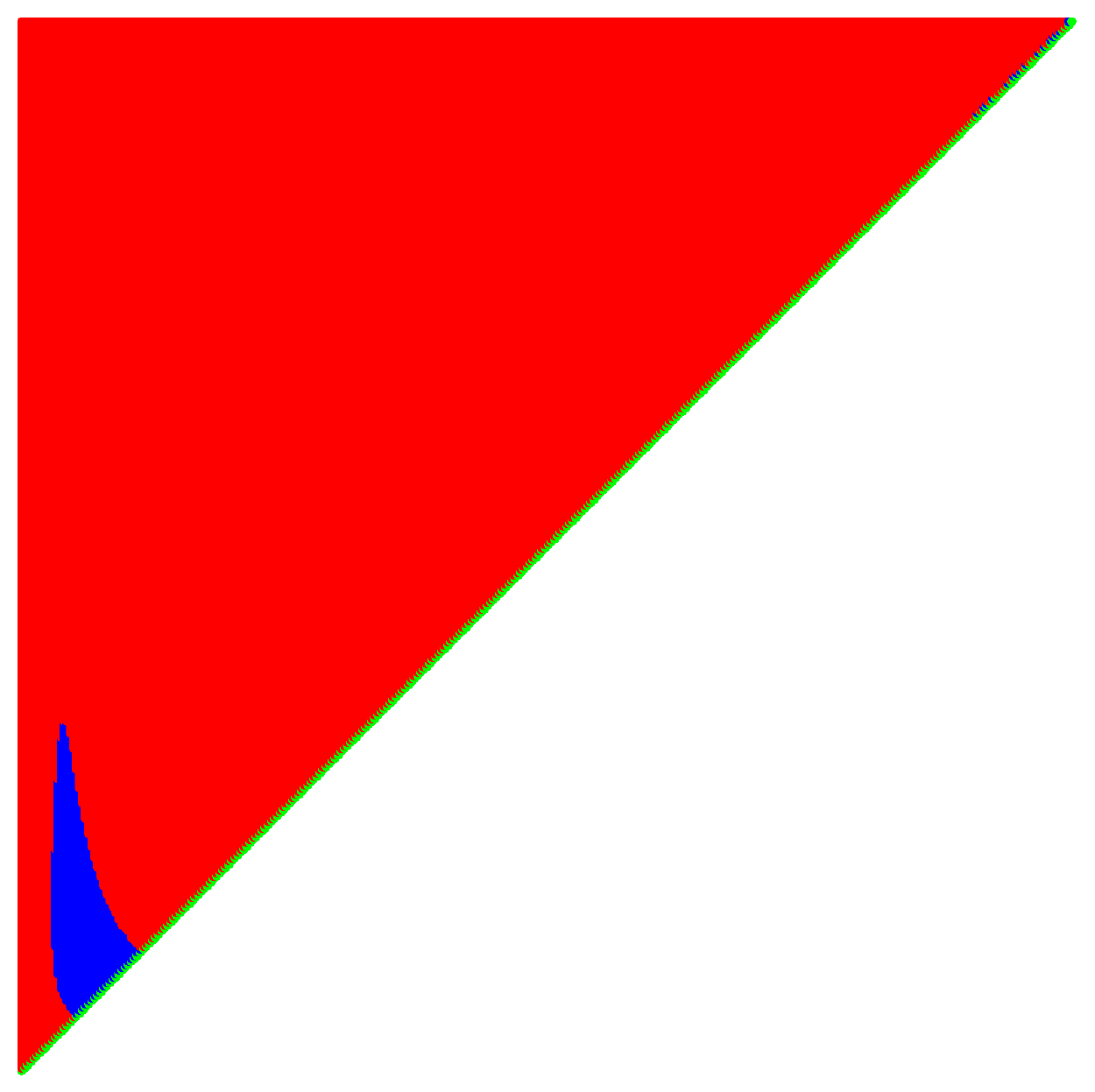}
  \captionof{figure}{Face(2) of a curve in 3-space.}
  \label{Fig:test1}
\end{minipage}%
\begin{minipage}{.5\textwidth}
  \centering
 \includegraphics[width=.4\linewidth]{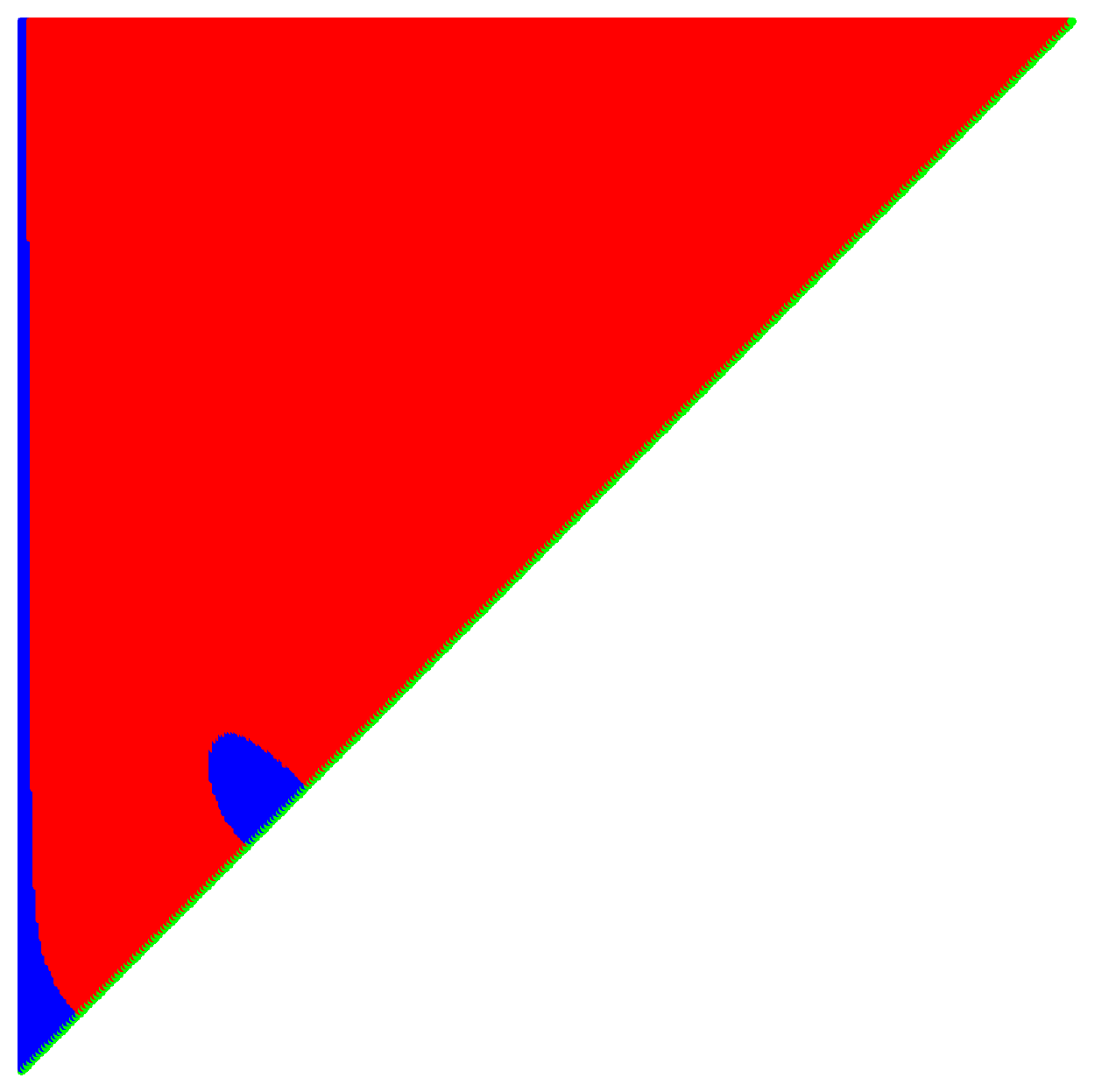}
  \captionof{figure}{Face(3) with initial point as one vertex of the 3-face for a curve in 4-space.}
  \label{fig:test2}
\end{minipage}
\end{figure}
\begin{lemma}\rm
A dynamical system $\dot{\textbf{x}}=\textbf{f(x)}$ where each $f_i$ is a polynomial in $s$ variables arises from a CRN with mass-action kinetics if and only if every monomial in $f_i$ with negative coefficient is divisible by $x_i$ for all $i\in\{1,2,\ldots,s\}.$
\end{lemma}

By this lemma there exists a chemical reaction network for each of the systems in the examples below.

\begin{example}\rm
Consider the following system with initial point as $x_{0}=(10,8,9,2)$,
\begin{equation}
\begin{split}
\dot{x_1} & =-2x_1^2 - 6x_1x_4 + 10x_3x_4 \\
\dot{x_2} & =x_1^2 - 8x_2x_3\\
\dot{x_3} & =x_1^2 + 6x_1x_4 - 9x_3x_4 \\
\dot{x_4} & = 8x_2x_3 - x_3x_4.
\end{split}
\end{equation}

The solution of this system lies in stoichiometry 
subspace of dimension 3 and hence, the convex hull has dimension 3. To find the curve of Face(2),
we consider the matrix given by
\begin{equation}\label{eq:facevset4}
\begin{pmatrix}
               \,  1  &     1 &   0  &   0   \\
             \, c_{3i}  &   c_{3j}  & c^{\prime}_{3i}   &    c^{\prime}_{3j}  
               
\end{pmatrix}
\end{equation}
as in (\ref{eq:facevset2}) for $i,j\in \{2,3, \ldots , 2000 \}$ and $i\leq j$. We plot this in \cref{Fig:test1} where blue and red represents that the sign of the determinant is negative and positive, respectively. The separating boundary of the red and blue area represents the Face(2) of this system.
$\hfill\square$
\end{example}

\bigskip
Next, we consider a curve with a 4-dimensional convex body.

\begin{example}\rm
Consider the following system with initial point as $x_{0}=(5,8,6,2)$,
\begin{equation}
\begin{split}
\dot{x_1} & =-10x_1^2 + 12x_2x_3 + 6x_3^2 + 4x_3x_4 - 5x_1 \\
\dot{x_2} & =2x_1^2 - 8x_2x_3 + x_1\\
\dot{x_3} & =8x_1^2 - 8x_2x_3 - 6x_3^2 + 5x_1 \\
\dot{x_4} & = - 8x_3x_4 + 4x_1.
\end{split}
\end{equation}

For this system we consider the representation of faces that has initial point as always one of the vertex. This is given by considering the matrix in \cref{eq:facevset6} with the initial point as the boundary point. The boundary of the red and the blue area in \cref{fig:test2} gives the curve describing the Face(3) of the system such that every point on this curve represents the face of the convex hull such that initial point is one of the three vertices of that face.
\begin{equation}\label{eq:facevset6}
\begin{pmatrix}
               \,  1  &     1 &1 &    0  &   0   \\
             \, c_{3i}  &   c_{3j}  &x_0&  c^{\prime}_{3i}   &    c^{\prime}_{3j}  
               
\end{pmatrix}
\end{equation}
and for $i,j\in \{2, \ldots , 2000 \}$ and $i\leq j$.
$\hfill\square$
\end{example}
\begin{figure}[t]
\begin{center}
\vspace{-0.2in}
  \includegraphics[width=3.5cm]{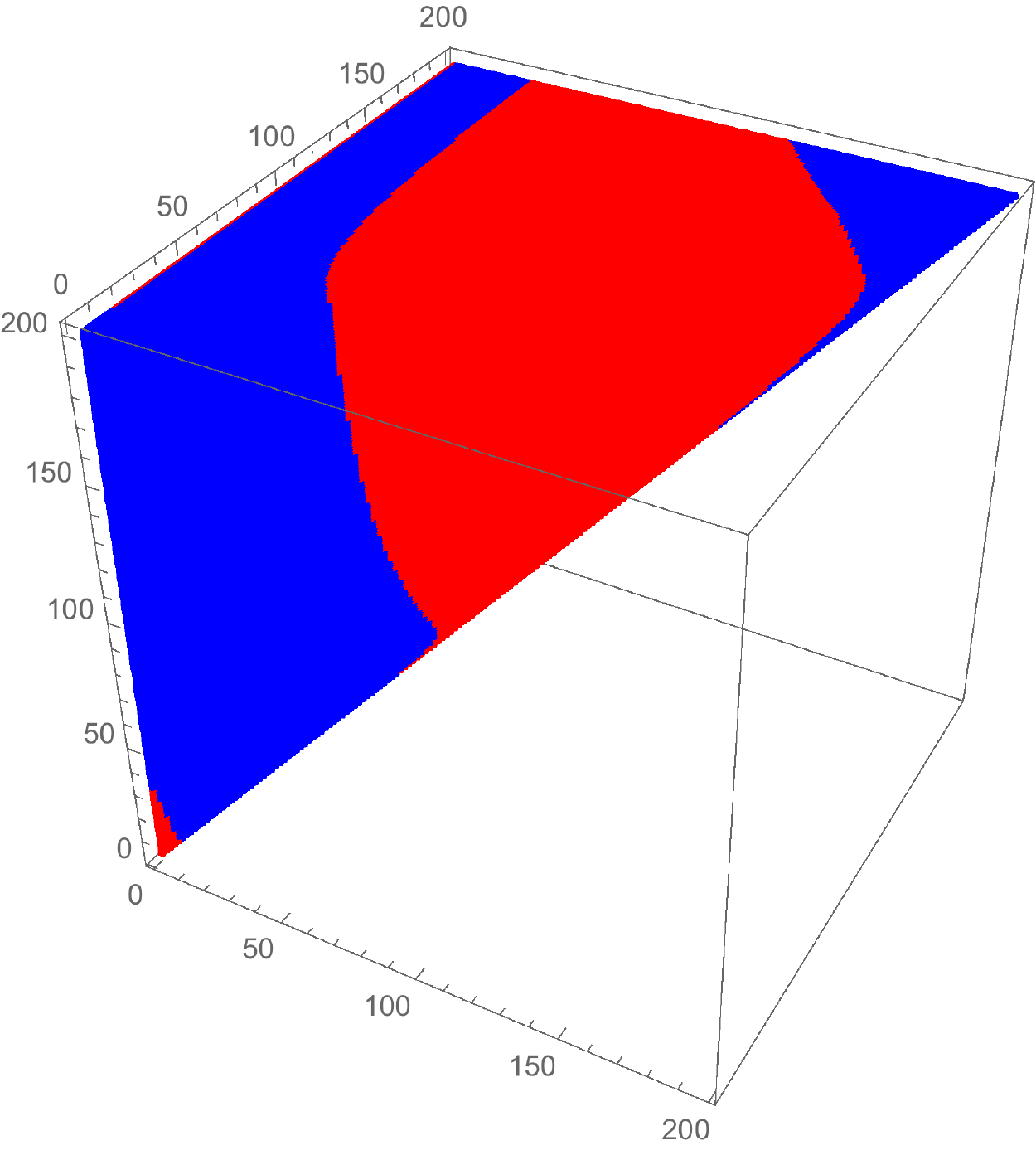}\hspace*{2.3em}
  \includegraphics[width=3.5cm]{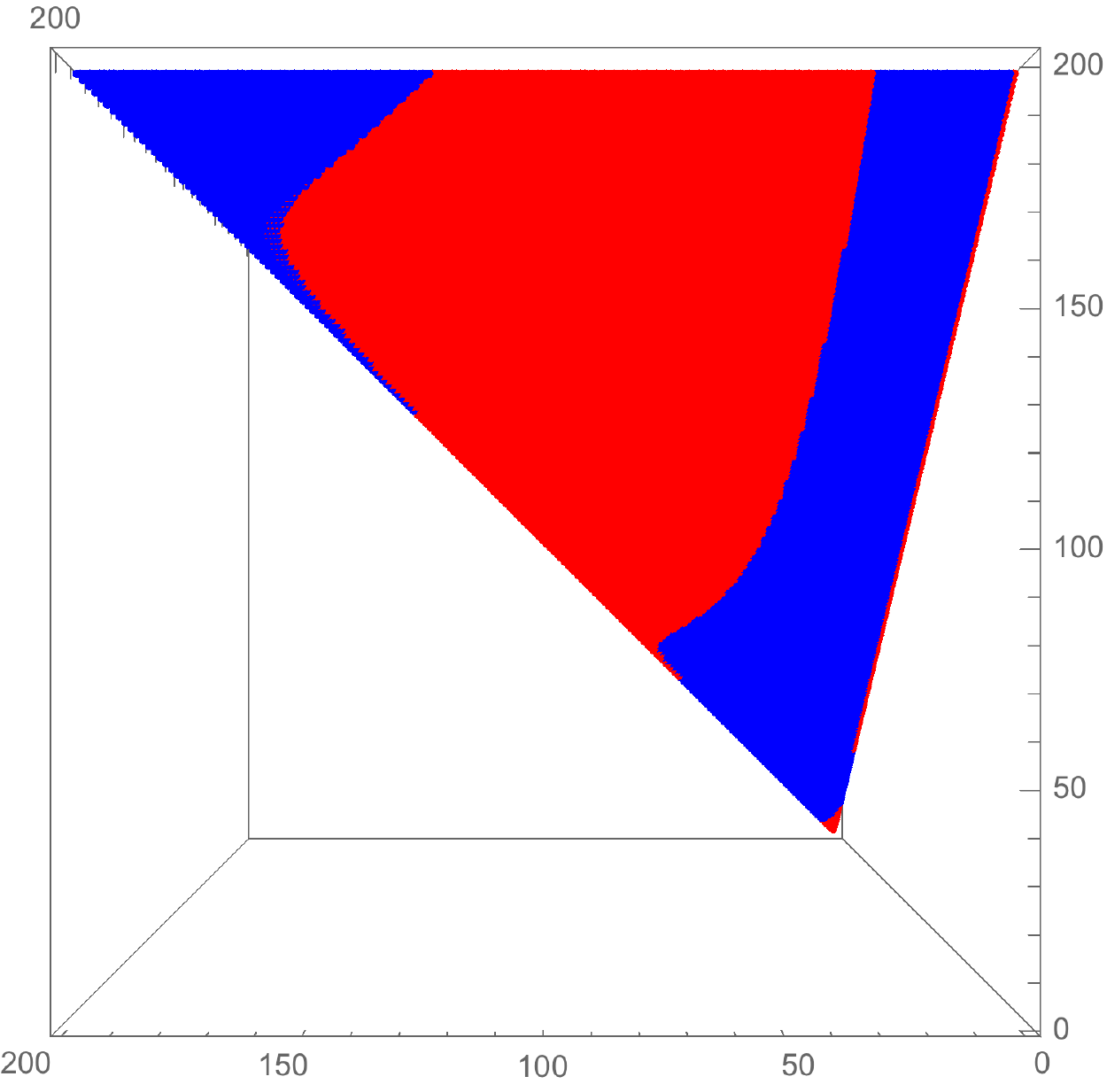}\hspace*{2.3em}
  \includegraphics[width=3.5cm]{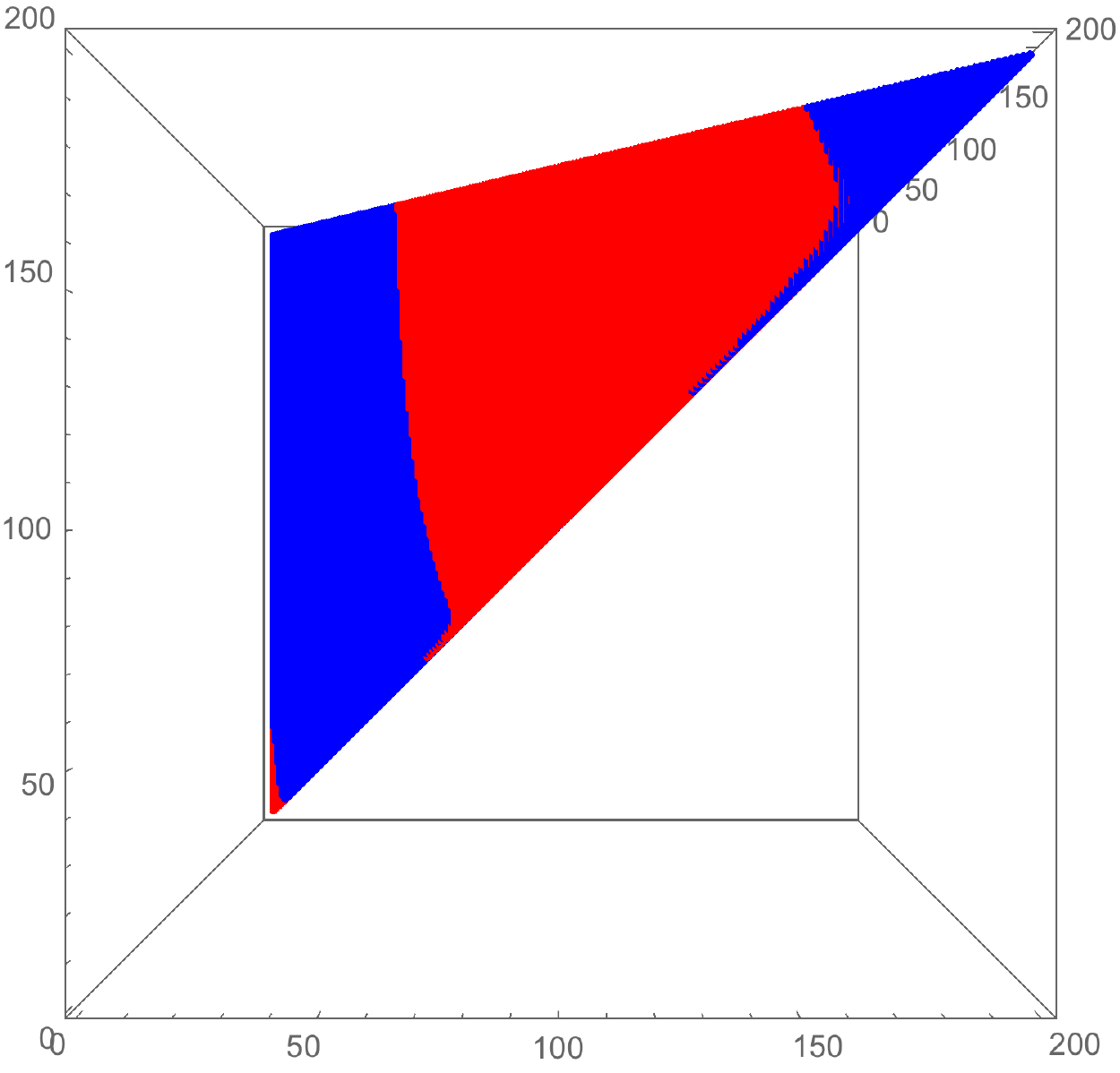}
\vspace{-0.3in}
\end{center}
    \caption{\label{fig:drei} 
Face(3) of a 5-dimensional convex body}
         \end{figure}
The following example will depict the Face(3) of a trajectory in 5-dimensions.
\begin{example}\rm The system given by
\begin{equation}
\begin{split}
\dot{x_1} & =4x_3x_4x_6 - 8x_1x_6^2 + 2x_2^2 + 4x_3x_5 \\
\dot{x_2} & =-10x_2^2x_4 + 4x_3x_4x_6 + 4x_1x_6^2 - 12x_2^2 + 6x_6^2\\
\dot{x_3} & =5x_2^2x_4 - 6x_3x_4x_6 + 6x_1x_6^2 - 4x_3x_5 + 2x_6^2 \\
\dot{x_4} & =-4x_2^2x_4 - 4x_3x_4x_6 + 2x_1x_6^2 + 2x_6^2\\
\dot{x_5} & =4x_2^2x_4 + 4x_1x_6^2 - 4x_3x_5 \\
\dot{x_6} & = x_2^2x_4 + 2x_3x_4x_6 - 14x_1x_6^2 + 12x_2^2 + 8x_3x_5 - 8x_6^2
\end{split}
\end{equation}
has stoichiometry space of dimension 5. The \cref{fig:drei} shows the sign of determinant of 
\begin{equation}
\begin{pmatrix}
               \,  1  &     1 &1 &    0  &   0  &0 \\
             \, c_{5i}  &   c_{5j}  &c_{5k}&  c^{\prime}_{5i}   &    c^{\prime}_{5j} &    c^{\prime}_{5k}  
               
\end{pmatrix}
\end{equation}
for $i,j,k\in \{2,\ldots , 200\}$ and $i<j<k.$
$\hfill\square$
\end{example}

Using this adaptation 
for understanding the convex hulls is not sufficient. This approach when 
applied to the trajectories could not be used to give a representation of all the 
faces and therefore, for the curves coming from a dynamical system this adaptation could not 
give a general description. Clearly, there are some rich veins of research here which can be pursued much further.

\section{Discussion}

This work is motivated by an optimisation problem in chemistry, namely the one of finding the most cost-efficient reactor. It is of great interest for industrial chemists to find the optimum reactor while improving the reaction efficiency. The feasible set of this problem is a convex object. Since this region is a geometric object this problem lies on the interface of chemistry and mathematics. The formalism that we have established in this paper now provides the basis to describe and explore the properties of these convex sets coming from chemistry, using the language of algebraic geometry and characterise them. For certain linear systems we could express this 
convex object as a spectrahedral shadow. However, by results due to Claus Scheiderer in \cite{Sch1} it is not 
possible to express every convex object as a spectrahedral shadow. 

There are a number of intriguing new stream of research coming out of our analysis. We conjectured that attainable region of weakly reversible systems with linkage class one is the convex hull of the trajectory. In the future, we hope to work towards resolving this conjecture and give a representation of 
the same. One possible way of tackling this problem could be via an approximation of this region by 
the semidefinite representable sets. As a second step, it would also be very interesting to study the 
systems where the attainable region is larger than the convex hull of the trajectory. In particular, 
understanding the attainable regions of the systems with multistationary points 
may prove to be insightful. Giving a representation by way of studying the faces is yet another interesting problem for convex hulls coming from such trajectories.
\bigskip

\begin{small}
\noindent
{\bf Acknowledgements.}
The author would like to express her gratitude to Bernd Sturmfels for suggesting the problem, and providing valuable advice and support along the way. She is grateful to Christiane G$\ddot{\text{o}}$rgen for feedback and useful comments on the draft of the manuscript. She is thankful to Amir Ali Ahmadi, Anne Shiu and Cynthia Vinzant for their help and useful discussions.
The author was funded by the
International Max Planck Research School {\em Mathematics in the
  Sciences} (IMPRS). 
\end{small}

 \begin{small}
 
 \end{small}

\end{document}